\documentclass[final,1p,times]{elsarticle}
\usepackage{amssymb}
\usepackage{amsmath}
\usepackage{amsthm}
\usepackage[utf8]{inputenc}
\usepackage{marginnote}
\usepackage{amsfonts}
\usepackage{textcomp}
\usepackage{color,soul}
\usepackage{placeins}
\usepackage{tikz}
\usepackage{makecell}
\usepackage{amssymb}
\usepackage{graphicx}
\usepackage{hyperref}
\usepackage[none]{hyphenat}
\usepackage{lipsum}
\newtheorem{thm}{Theorem}[section]

\newtheorem{example}[thm]{Example}
\newdefinition{defn}[thm]{Definition}
\allowdisplaybreaks

\journal{IEEE Access}
\begin{document}
	\begin{frontmatter}
	\title{On the paper  “Optimal dual frames of probabilistic erasures" }
        \author{Shankhadeep Mondal}
		\ead{shankhadeep.mondal@ucf.edu}
		\author{Ram Narayan Mohapatra}
		\ead{Ram.Mohapatra@ucf.edu}

		\address{University of Central Florida}
		\cortext[Ram]{Corresponding Author: Ram.Mohapatra@ucf.edu}

\begin{abstract}
In the paper “Optimal Dual Frames for Probabilistic Erasures” \cite{li1}, the authors have given conditions under which the canonical dual is claimed to be the unique probability optimal dual for $1-$erasure reconstruction. In this paper, we demonstrate via counterexamples that the conditions provided are not sufficient to guarantee uniqueness. We also noticed a mistake in the proof of the theorem and proved the correct version of the theorem with a stronger but valid condition. Furthermore, we show that the corollary asserting uniqueness for a tight frame assumption is also incorrect. Our results refine the understanding of probability optimal dual frame constructions and offer a more complete characterization of the $1-$erasure probability optimal duals.
\end{abstract}

		
	\end{frontmatter}
	
\section{Introduction}

Frames offer redundant yet stable representations in Hilbert spaces, making them highly effective in applications such as communications, signal processing, and error correction. Their redundancy ensures robustness against coefficient loss, which enhances reliability in applications subject to erasures. 

In \cite{li1}, Li et al. investigated the problem of recovering signals from probabilistically erased frame coefficients. They provided sufficient conditions under which the canonical dual frame is optimal and proposed criteria for selecting alternate duals based on probabilistic weights. These results are significant for practical scenarios where frame coefficients may be lost with varying probabilities.

While the contributions in \cite{li1} are impactful, we identify flaws in their main uniqueness result. Specifically, Theorem~1 and its corollaries claim uniqueness of the canonical dual to be probability optimal under certain assumptions. However, we show that these conditions are insufficient to guarantee uniqueness by constructing counterexamples where multiple duals yield the same optimal reconstruction error. We then propose a corrected version of the theorem with stronger assumptions.

We also examine Corollary~1 and Corollary~2 from \cite{li1}, showing that they do not guarantee uniqueness as claimed. In both cases, we provide explicit counterexamples and characterize the full set of optimal duals.

This paper is organized as follows: Section~2 briefly reviews key definitions and the probabilistic erasure model. Section~3 presents counterexamples to the original results and a corrected theorem, along with a complete description of the optimal dual set in the tight frame case.

\medskip

{\small
\noindent \textbf{Related Works:} For broader context on optimal dual frames and erasures, we refer the reader to foundational works such as \cite{holm, leng3, bodm1, peh, sali, jerr, goya, jins, jerr, dev, ara, shan2}, and on probabilistic models of erasures in \cite{leng, li}. These works lay the groundwork for the investigation of dual frame robustness, including deterministic and probabilistic measures of error. Our analysis continues along this path, specifically refining the framework introduced in \cite{li1}.
}

\section{Preliminaries}

Let $\mathcal{H}$ be an $N$-dimensional Hilbert space.  A finite sequence \( F = \{f_i\}_{i=1}^M \) in a Hilbert space \( \mathcal{H} \) is called a \emph{frame} if there exist constants \( A, B > 0 \) such that for all \( f \in \mathcal{H} \),
\[
A\|f\|^2 \leq \sum_{i=1}^M |\langle f, f_i \rangle|^2 \leq B\|f\|^2.
\]
The constants \( A \) and \( B \) are known as the lower and upper frame bounds. If \( A = B \), the frame is called \emph{tight}, and if \( A = B = 1 \), it is called a \emph{Parseval frame}. Associated with a frame \( F \) is the \emph{analysis operator} \( \Theta: \mathcal{H} \to \mathbb{C}^M \), defined by \( \Theta_{F}(f) = \{\langle f, f_i \rangle\}_{i=1}^M \), and its adjoint, the \emph{synthesis operator} \( \Theta_F^*: \mathbb{C}^M \to \mathcal{H} \), given by \( \Theta_F^*(\{c_i\}) = \sum_{i=1}^M c_i f_i \). The \emph{frame operator} \( S: \mathcal{H} \to \mathcal{H} \), defined by \( S = \Theta_F^* \Theta_F \), satisfies
\[
S f = \sum_{i=1}^N \langle f, f_i \rangle f_i,
\]
and is linear, positive, self-adjoint, and invertible.

A sequence \( G = \{g_i\}_{i=1}^M \) is called a \emph{dual frame} of \( F \) if every \( f \in \mathcal{H} \) can be reconstructed as
\[
f = \sum_{i=1}^N \langle f, f_i \rangle g_i = \sum_{i=1}^M \langle f, g_i \rangle f_i.
\]
This is equivalent to the condition \( \Theta_G^* \Theta_F = I \). In this case, \( F \) is also a dual of \( G \), and the pair \( (F, G) \) is referred to as an \((M, N)\) dual pair in the \( n \)-dimensional space \( \mathcal{H} \).

For a given dual frame $G = \{g_i\}_{i=1}^M$ of $F$, the \emph{probabilistic erasure error for $1-$erasure} under the weight number sequence $\{q_i\}_{i=1}^M$ is measured by
\[
r_1^q(F,G) = \max_{1 \leq i \leq M} q_i \cdot |\langle f_i, g_i \rangle|.
\]
A dual frame $G$ is said to be \textit{1-erasure probability optimal} if it minimizes $r_1^q(F,G)$ among all duals of $F$.

\section{Revisiting Theorem and Corollaries: Counterexamples and Corrections}
We take the weight number sequence $\{q_i\}_{i=1}^N$ to be a sequence of scalars where
\begin{equation}\label{eqn1}
    q_i = \dfrac{\sum\limits_{j=1}^N p_j}{\sum\limits_{j=1}^N p_j - p_i}. \dfrac{M-1}{N}.
\end{equation}
So $\{q_i\}_{i=1}^N$  satisfying

	\begin{enumerate}
	\item [{\em (i)}] $1\leq q_i < \infty,\; 1 \leq i \leq N ,$
	\item [{\em (ii)}] $\sum\limits_{i=1}^N \frac{1}{q_i} =n.$
        \item [{\em (iii)}] The larger is the number $p_i,$ the larger is the number $q_i$ for all $1 \leq i \leq N.$
\end{enumerate}
	
\subsection{Modified Theorem }

\noindent Let $F = \{f_i\}_{i=1}^M$ be a frame for a Hilbert space $\mathcal{H}$, and let $\{q_i\}_{i=1}^M$ be a sequence of positive weights. The following theorem provides a sufficient condition under which the canonical dual of $F$ is a 1-erasure probability optimal dual with respect to the given weights.
    
Let us define $c = \max \left\{ \sqrt{q_i} \left\| S_F^{-1/2} f_i \right\|:1 \leq i \leq M \right\},$ and let
\[
\Lambda_1 := \left\{ i \in \{1,\dots,N\} : \sqrt{q_i} \left\| S_F^{-1/2} f_i \right\| = c \right\}, \quad
\Lambda_2 := \{1, 2, \dots, N\} \setminus \Lambda_1.
\]
Define the subspaces $H_j := \operatorname{span} \{ f_i : i \in \Lambda_j :  \},$ for $j=1,2.$ 

\begin{thm}\label{thm3point1}
Let $F = \{f_i\}_{i=1}^M$ be a frame for $\mathcal{H}$ and let $\{q_i\}_{i=1}^M$ be  weight number sequence given by \eqref{eqn1}. If $H_1 \cap H_2 = \{0\}$, then the canonical dual $S_{F}^{-1}F$ is a $1-$erasure probability optimal dual of $F$ for $1-$erasure.
\end{thm}

\begin{proof}
Let $G = \{g_i\}_{i=1}^M$ be a dual of $F$ of the form $g_i = S_F^{-1}f_i + u_i \;$ for each $i$, where  $\{u_i\}_{i=1}^M$ satisfy
\[
\sum_{i=1}^M \langle f, u_i \rangle f_i = 0 \quad \text{for all } f \in \mathcal{H}.
\]
Using the assumption that $H_1 \cap H_2 = \{0\}$, it follows that
\[
\sum_{i \in \Lambda_1} \langle f, u_i \rangle f_i = 0 = \sum_{i \in \Lambda_2} \langle f, u_i \rangle f_i\quad \text{for all } f \in \mathcal{H}.
\]
This implies that
\[
\Theta_{F_1}^* \Theta_{U_1} = 0,
\]
where $F_1 = \{f_i\}_{i \in \Lambda_1}$ and $U_1 = \{u_i\}_{i \in \Lambda_1}$. Taking the trace, we obtain
\[
\operatorname{tr}\left(\Theta_{F_1}^* \Theta_{U_1}\right) = \operatorname{tr}\left(\Theta_{U_1} \Theta_{F_1}^*\right) = \sum_{i \in \Lambda_1} \langle f_i, u_i \rangle = 0.
\]
Therefore,
\begin{equation} \label{eq:tracezero}
\operatorname{Re} \left( \sum_{i \in \Lambda_1} \langle f_i, u_i \rangle \right) = 0.
\end{equation}

Now observe that:
\[
\max_{1 \leq i \leq N} q_i |\langle f_i, g_i \rangle|
\geq \max_{i \in \Lambda_1} q_i |\langle f_i, S_F^{-1} f_i \rangle + \langle f_i, u_i \rangle|
= \max_{i \in \Lambda_1} \left| c^2 + q_i \langle f_i, u_i \rangle \right|.
\]

If $\operatorname{Re} \langle f_i, u_i \rangle = 0$ for all $i \in \Lambda_1$, then $\max_{1 \leq i \leq M} q_i |\langle f_i, g_i \rangle| \geq c^2.$ Otherwise, if there exists $j \in \Lambda_1$ such that $\operatorname{Re} \langle f_j, u_j \rangle > 0$, then
\[
\max_{1 \leq i \leq N} q_i |\langle f_i, g_i \rangle| > c^2.
\]
On the other hand, if $\operatorname{Re} \langle f_j, u_j \rangle < 0$ for some $j \in \Lambda_1$, then from \eqref{eq:tracezero}, there must exist some $j' \in \Lambda_1$ such that $\operatorname{Re} \langle f_{j'}, u_{j'} \rangle > 0$, ensuring again that the maximum exceeds $c^2$.

Therefore, $\max_{1 \leq i \leq M} q_i |\langle f_i, g_i \rangle| \geq c^2,$ for any dual $G$ of $F,$ proving that the canonical dual is a $1-$erasure probability optimal dual of $F$.

\end{proof}

\subsection{Counterexample of Theorem~1}

In this section, we present a counterexample demonstrating that Theorem~1 in the paper \cite{li1} is not valid as stated. The theorem claims that if $F = \{f_i\}_{i=1}^M$ is a frame for $\mathcal{H}$, the subspaces $H_1 \cap H_2 = \{0\}$, and the set $\{f_i\}_{i \in \Lambda_2}$ is linearly independent, then the canonical dual $\{S_{F}^{-1}f_i\}_{i=1}^M$ is the \emph{unique} 1-erasure probability optimal dual of $F$. 

We construct a frame $F$ and a weight sequence $\{q_i\}$ that satisfy all hypotheses of the theorem. We then show that, in addition to the canonical dual, there exists a continuum of distinct duals  which are also $1-$erasure probabilistic optimal. Therefore, the uniqueness part of the theorem fails, despite the other conditions being met.

\begin{example}
Let $F = \{f_1,f_2,f_3, f_4, f_5\}$ be a frame for $\mathbb{C}^4$ given by
\[
\begin{aligned}
f_1 = \begin{bmatrix} 1 \\ 0 \\ 0 \\ 0 \end{bmatrix}, \quad
f_2 = \begin{bmatrix} 1 \\ 1 \\ 0 \\ 0 \end{bmatrix}, \quad
f_3 = \begin{bmatrix} 1 \\ 2 \\ 0 \\ 0 \end{bmatrix}, \quad
f_4 = \begin{bmatrix} 0 \\ 0 \\ 1 \\ 2 \end{bmatrix}, \quad
f_5 = \begin{bmatrix} 0 \\ 0 \\ 2 \\ 1 \end{bmatrix},
\end{aligned}
\]
with the weight sequence $\{q_i\}_{i=1}^5 = \left\{ \frac{7}{5}, \frac{7}{2}, 1, 1, 1 \right\}$.
\end{example}
The frame operator of $F$ is
\[
S = \begin{bmatrix}
3 & 3 & 0 & 0 \\
3 & 5 & 0 & 0 \\
0 & 0 & 5 & 4 \\
0 & 0 & 4 & 5
\end{bmatrix}.
\]
The canonical dual of $F$ is given by $S^{-1}F = \{S^{-1}f_i\}_{i=1}^5,$ where
\[
\begin{aligned}
S^{-1}f_1 = \begin{bmatrix} \frac{5}{6} \\ -\frac{1}{2} \\ 0 \\ 0 \end{bmatrix}, \quad
S^{-1}f_2 = \begin{bmatrix} \frac{1}{3} \\ 0 \\ 0 \\ 0 \end{bmatrix}, \quad
S^{-1}f_3 = \begin{bmatrix} \frac{1}{6} \\ \frac{1}{2} \\ 0 \\ 0 \end{bmatrix}, \quad
S^{-1}f_4 = \begin{bmatrix} 0 \\ 0 \\ -\frac{1}{3} \\ \frac{2}{3} \end{bmatrix}, \quad
S^{-1}f_5 = \begin{bmatrix} 0 \\ 0 \\ \frac{2}{3} \\ -\frac{1}{3} \end{bmatrix}.
\end{aligned}
\]

We compute the weighted inner products:
\[
q_1\langle f_1, S^{-1}f_1 \rangle = q_2\langle f_2, S^{-1}f_2 \rangle = q_3\langle f_3, S^{-1}f_3 \rangle = \frac{7}{6}, \quad
q_4\langle f_4, S^{-1}f_4 \rangle = q_5\langle f_5, S^{-1}f_5 \rangle = 1.
\]

Thus, $\sqrt{q_1}\left\| S^{-1/2} f_1 \right\|= \sqrt{q_1\langle f_1, S^{-1}f_1 \rangle} = \sqrt{\frac{7}{6}},\quad \sqrt{q_2}\left\| S^{-1/2} f_2 \right\|= \sqrt{q_3}\left\| S^{-1/2} f_3 \right\| =  \sqrt{\frac{7}{6}}, \quad \sqrt{q_4}\left\| S^{-1/2} f_4 \right\| = \sqrt{q_5}\left\| S^{-1/2} f_5 \right\| = 1.$ Therefore, $c= \sqrt{\frac{7}{6}} , \;\Lambda_1 = \{1,2,3\}, \,\Lambda_2 = \{4,5\}.$ Note that $\{f_4, f_5\}$ is linearly independent. By Theorem~\ref{thm3point1}, the canonical dual is a $1$-erasure probability optimal dual of $F$.

Now consider another dual $G = \{g_i\}_{i=1}^5 = \{S^{-1}f_i +u_i\}_{i=1}^5$ of $F$ defined by
\[
\begin{aligned}
g_1 = \begin{bmatrix} \frac{5}{6} \\ -\frac{1}{2} \\ 1 \\ 1 \end{bmatrix} , \quad
g_2 = \begin{bmatrix} \frac{1}{3} \\ 0 \\ -2 \\ -2 \end{bmatrix}, \quad
g_3 = \begin{bmatrix} \frac{1}{6} \\ \frac{1}{2} \\ 1 \\ 1 \end{bmatrix}, \quad
g_4 = \begin{bmatrix} 0 \\ 0 \\ -\frac{1}{3} \\ \frac{2}{3} \end{bmatrix} , \quad
g_5 = \begin{bmatrix} 0 \\ 0 \\ \frac{2}{3} \\ -\frac{1}{3} \end{bmatrix} .
\end{aligned}
\]

It is easy to verify that $G$ is a dual of $F$, since $\sum_{i=1}^5 \langle f, f_i \rangle u_i = 0$ for all $f \in \mathcal{H}$, where $u_i = g_i - S_F^{-1} f_i$.

Moreover,
\[
q_1\langle f_1, g_1 \rangle = q_2\langle f_2, g_2 \rangle = q_3\langle f_3, g_3 \rangle = \frac{7}{6}, \quad
q_4\langle f_4, g_4 \rangle = q_5\langle f_5, g_5 \rangle = 1.
\]
Thus, $G$ is also a $1$-erasure probability optimal dual of $F$.

In fact, any dual of $F$ of the form
\[
\{g_i\}_{i=1}^5 = \left\{
\begin{bmatrix} \frac{5}{6} \\ -\frac{1}{2} \\ 0 \\ 0 \end{bmatrix} + \begin{bmatrix} 0 \\ 0 \\ \alpha_1 \\ \alpha_2 \end{bmatrix},
\begin{bmatrix} \frac{1}{3} \\ 0 \\ 0 \\ 0 \end{bmatrix} + \begin{bmatrix} 0 \\ 0 \\ -2\alpha_1 \\ -2\alpha_2 \end{bmatrix},
\begin{bmatrix} \frac{1}{6} \\ \frac{1}{2} \\ 0 \\ 0 \end{bmatrix} + \begin{bmatrix} 0 \\ 0 \\ \alpha_1 \\ \alpha_2 \end{bmatrix},
\begin{bmatrix} 0 \\ 0 \\ -\frac{1}{3} \\ \frac{2}{3} \end{bmatrix},
\begin{bmatrix} 0 \\ 0 \\ \frac{2}{3} \\ -\frac{1}{3} \end{bmatrix}
\right\}, \quad \alpha_1, \alpha_2 \in \mathbb{C},
\]
is a $1$-erasure probability optimal dual of $F$. Hence, the set of all such duals is uncountable.

\subsection{Counterexample of Corollary 1 and 2 :}
We now present a counterexample to Corollary~1 from the paper \cite{li1}. The corollary asserts that for a given tight frame $F = \{f_i\}_{i=1}^M$ for a Hilbert space $\mathcal{H}$ and a weight sequence $\{q_i\}_{i=1}^M$, the canonical dual $\{S^{-1}f_i\}_{i=1}^M$ is the \emph{unique} 1-erasure probabilistically optimal dual if and only if the set $\Lambda_2$  is empty.

In the following example, we construct a tight frame $F$ for $\mathbb{C}^2$ with frame bound $3$ and a weight sequence $\{q_i\}_{i=1}^4$ satisfying the conditions of the corollary. We show that the canonical dual is indeed $1-$erasure probabilistically optimal, but not uniquely so. Specifically, we identify an uncountable family of duals of $F$ that also attain the same optimality criterion. This directly contradicts the uniqueness claim in Corollary~1, even under the assumption that $\Lambda_2 = \emptyset$.

\begin{example}\label{eg3point3}
	Consider the tight frame $F= \{f_i\}_{i=1}^4 = \left\{\left[\begin{array}{l}
		1 \\0
	\end{array}\right],  \left[\begin{array}{l}
	0 \\ 1
\end{array}\right] , \left[\begin{array}{l}
1 \\ 1
\end{array}\right] , \left[\begin{array}{l}
1 \\ -1
\end{array}\right] \right\} $ for $\mathbb{C}^2$ with frame bound $3$ 
	and the weight number sequence  $\{q_i\}_{i=1}^4 = \left\{ 3,3,\frac{3}{2} , \frac{3}{2}\right\}.$
\end{example}

In this case, we have 
$$S_{F}^{-1} F = \left\{ \left[\begin{array}{r} \frac{1}{3} \\~\\ 0 \end{array}\right], \left[\begin{array}{r} 0 \\~\\ \frac{1}{3} \end{array}\right], \left[\begin{array}{l} \frac{1}{3} \\~\\ \frac{1}{3}  \end{array}\right],  \left[\begin{array}{r} \frac{1}{3} \\~\\ -\frac{1}{3} \end{array}\right] \; \right\} $$
and $r_{1}^q (F,S_{F}^{-1}F) = 1.$ For any dual $G$ of $F,\;r_{1}^q (F,G) \geq 1.$ Suppose not. Then, $ q_i |\langle f_i , g_i \rangle | < 1 ,$ for all $1 \leq i \leq M.$ This implies that  $N= \displaystyle{\sum_{i=1}^M \langle f_i , g_i \rangle \leq \sum_{i=1}^M |\langle f_i , g_i \rangle | < \sum_{i=1}^M \frac{1}{q_i} =N },$ \; which is absurd. Hence, the canonical dual is a $1-$erasure probability optimal dual of $F$. Explicitly, any dual $G'$ of $F$ can be written as 

\begin{align}
	G''= \{ g''_i \}_{i=1}^4 = \left\{ \begin{bmatrix} \frac{1}{3} - \alpha -\beta  \\ -\gamma - \delta \end{bmatrix}, \begin{bmatrix} -\alpha +\beta \\ \frac{1}{3} -\gamma +\delta  \end{bmatrix}, \begin{bmatrix} \frac{1}{3} + \alpha \\ \frac{1}{3} +\gamma  \end{bmatrix}, \begin{bmatrix} \frac{1}{3} +\beta\\ -\frac{1}{3} +\delta \end{bmatrix}  \right\},
	\end{align}
where $\alpha, \beta, \gamma, \delta \in \mathbb{C}.$\\

Let \( F' = \{f'_i\}_{i=1}^M \) be a frame for a Hilbert space and \( G' = \{g'_i\}_{i=1}^M \) be a dual frame of \( F' \). Suppose  \( r_1^q(F', G') = 1 \). Then it must hold that
\[
q_i \langle f'_i, g'_i \rangle = 1, \quad \text{for all } 1 \leq i \leq M.
\]
To see this, assume by contradiction that there exists some \( j \in \{1, 2, \dots, M\} \) such that
\[
q_j \left| \langle f'_j, g'_j \rangle \right| < 1.
\]
Then 
\[
N = \left| \sum_{i=1}^M \langle f'_i, g'_i \rangle \right| \leq \sum_{i=1}^M \left| \langle f'_i, g'_i \rangle \right| < \sum_{i=1}^M \frac{1}{q_i} = N,
\]
which leads to a contradiction. Hence, we conclude that
\[
q_i \left| \langle f'_i, g'_i \rangle \right| = 1 \quad \text{for all } i = 1, \dots, M.
\]

It now remains to show that \( \langle f'_i, g'_i \rangle \geq 0 \) for all \( i \). Observe that
\[
\sum_{i=1}^M \langle f'_i, g'_i \rangle = N = \sum_{i=1}^M \frac{1}{q_i} = \sum_{i=1}^M \left| \langle f'_i, g'_i \rangle \right|.
\]
Let us write \( \langle f'_i, g'_i \rangle = a_i + ib_i \), where \( a_i, b_i \in \mathbb{R} \). Then the above equality implies
\[
\sum_{i=1}^M a_i = \sum_{i=1}^M \sqrt{a_i^2 + b_i^2}.
\]
This is possible only if \( b_i = 0 \) and \( a_i \geq 0 \) for all \( i \), which gives \( \langle f'_i, g'_i \rangle \in \mathbb{R}_{\geq 0} ,\; 1\leq i \leq M\).


Now if $G''= \{ g''_i \}_{i=1}^4$ is a $1-$erasure probabilistic optimal dual of $F$ then $q_i \langle f_i,g''_i \rangle =1,\;1\leq i \leq 4.$ This leads to the following system of equations:
\begin{equation}
	\begin{array}{ccl}
		  \alpha + \beta &=& 0\\
		 -\gamma + \delta  &=&0 \\
		 \alpha +\gamma &=& 0\\
		  \beta -\delta  &=& 0,
	\end{array}
\end{equation}
for which the solution set is $\{(-w,w,w,w): w \in \mathbb{C}\}.$ Therefore, the set
$$ \left\{ \left\{ \left[\begin{array}{r} \frac{1}{3}  \\~\\ -2w \end{array}\right], \left[\begin{array}{r} 2w \\~\\ \frac{1}{3}  \end{array}\right], \left[\begin{array}{r} \frac{1}{3} -w \\~\\ \frac{1}{3}  + w \end{array}\right] , \left[\begin{array}{r} \frac{1}{3} + w \\~\\ -\frac{1}{3}  + w \end{array}\right]      \right\}: \; w \in \mathbb{C} \right\}, $$
is the set of $1-$erasure probabilistic  optimal duals of $F.$\\~\\

Now we provide a counterexample to Corollary~2 of \cite{li1}, which claims that the canonical dual is the \emph{unique} 1-erasure probability optimal dual if and only if $\sqrt{q_i} \|f_i\|$ is constant for all $i$.

\begin{example}
Consider the same tight frame and weight sequence as in Example~3.3:
\[
F = \left\{ 
\begin{bmatrix}1\\0\end{bmatrix}, 
\begin{bmatrix}0\\1\end{bmatrix}, 
\begin{bmatrix}1\\1\end{bmatrix}, 
\begin{bmatrix}1\\-1\end{bmatrix}
\right\}, \quad 
\{q_i\}_{i=1}^4 = \left\{3, 3, \tfrac{3}{2}, \tfrac{3}{2} \right\}.
\]
Then,
\[
\sqrt{q_i} \|f_i\| = \sqrt{3}, \quad \text{for all } i=1,\dots,4.
\]
So the hypothesis of Corollary~2 is satisfied. However, as shown in Example~3.3, the set of all duals $G = \{g_i\}_{i=1}^4$ for which $q_i \langle f_i, g_i \rangle = 1$ for all $i$ is uncountable. Thus, the uniqueness conclusion of the canonical dual in Corollary~2 fails.
\end{example}

\vspace{1cm}

\noindent
	{\bf Acknowledgments:} 
	The authors are grateful to the Mohapatra Family Foundation and the College of Graduate Studies of the University of Central Florida for their support during this research.

\bibliographystyle{amsplain}

\begin{thebibliography}{10}













\bibitem{bodm1}
B.~G. Bodmann and V.~I. Paulsen,
\newblock {\em Frames, graphs and erasures}.
\newblock {Linear Algebra Appl.} \textbf{404} (2005), 118--146.


\bibitem{dev}
P.~Devaraj and S.~Mondal,
\newblock {\em Spectrally optimal dual frames for erasures}.
\newblock {Proc. Indian Acad. Sci. (Math. Sci.)} \textbf{133} (2023).

\bibitem{goya}
V.~K. Goyal, J.~Kova{\v{c}}evi{\'c} and J.~A. Kelner,
\newblock {\em Quantized frame expansions with erasures}.
\newblock {Appl. Comput. Harmon. Anal.} \textbf{10} (2001), 203--233.








\bibitem{holm}
R.~B. Holmes and V.~I. Paulsen,
\newblock {\em Optimal frames for erasures}.
\newblock {Linear Algebra Appl.} \textbf{377} (2004), 31--51.










 






\bibitem{jins}
J.~Leng and D.~Han,
\newblock {\em Optimal dual frames for erasures {II}}.
\newblock {Linear Algebra Appl.} \textbf{435} (2011), 1464--1472.

\bibitem{leng3}
J.~Leng, D.~Han and T.~Huang,
\newblock {\em Optimal dual frames for communication coding with probabilistic
  erasures}.
\newblock {IEEE Trans. Signal Process.} \textbf{59} (2011), 5380--5389.

\bibitem{leng}
J.~Leng, D.~Han and T.~Huang,
\newblock {\em Probability modelled optimal frames for erasures}.
\newblock {Linear Algebra Appl.} \textbf{438} (2013), 4222--4236.


\bibitem{li1}
D.~Li, J.~Leng and M.~He,
\newblock {\em Optimal dual frames for probabilistic erasures}.
\newblock {IEEE Access} \textbf{7} (2019), 2774--2781.

\bibitem{li}
D. Li,  and J. Leng,  and T. Huang,  and Q. Gao,
\newblock {\em Frame expansions with probabilistic erasures}.
\newblock {Digital Signal Processing}  (2028), 75--82.

\bibitem{jerr}
J.~Lopez and D.~Han,
\newblock {\em Optimal dual frames for erasures}.
\newblock {Linear Algebra Appl.} \textbf{432} (2010), 471--482.

\bibitem{sali}
S.~Pehlivan, D.~Han and R.~Mohapatra,
\newblock {\em Linearly connected sequences and spectrally optimal dual frames for
  erasures}.
\newblock {J. Funct. Anal.} \textbf{265} (2013), 2855--2876.

\bibitem{ara}
Fahimeh Arabyani-Neyshaburi, Ali Akbar Arefijamaal, and Ghadir Sadeghi, \emph{Numerically and spectrally optimal dual frames in hilbert spaces}, Linear Algebra and its Applications,604 (2020), 52–71.

\bibitem{peh}
S.~Pehlivan, D.~Han and R.~Mohapatra, \emph{Spectrally two-uniform
  frames for erasures}, Oper. Matrices \textbf{9}, no.~2, 383--399 (2015).


\bibitem{shan2}
S. Arati,  and P. Devaraj, P and Shankhadeep Mondal, \emph{Optimal dual pairs of frames for erasures}, Linear and Multilinear Algebra (2025).

\bibitem{shan}
Arati, S. and Devaraj, P. and Mondal, S., \emph{Optimal dual frames and dual pairs for probability modelled erasures}, Advances in Operator Theory, 52--71 (2024).

\end{thebibliography}

\end{document}